\documentclass[12pt,twoside]{siamart1116}

\usepackage[english]{babel}
\usepackage{graphicx,epstopdf,epsfig}
\usepackage{amsfonts,epsfig,fancyhdr,graphics, hyperref,amsmath,amssymb}

\usepackage{mathtools}
\usepackage{enumerate}  
\usepackage{url} 
\usepackage{color}
\definecolor{red}{rgb}{.8,0,0}

\definecolor{bblu}{rgb}{0,0,1}

\definecolor{purple}{rgb}{.8,0,1}

\newcommand{\HE}{Name of Handling Editor}
\newcommand{\DoS}{Month/Day/Year}
\newcommand{\DoA}{Month/Day/Year}
\newcommand{\CA}{Leonardo de Lima}

\newcommand{\Names}{Abiad, de Lima, Desai, Guo, Hogben, Madrid}
\newcommand{\Title}{Positive and Negative Square Energies of Graphs}

\newtheorem{remark}[theorem]{Remark}
\newtheorem{example}[theorem]{Example}
\newtheorem{question}[theorem]{Question}

\renewcommand{\theequation}{\thesection.\arabic{equation}}
\renewtheorem{theorem}{Theorem}[section]

 \newtheorem{conjecture}[theorem]{Conjecture}
 
 \numberwithin{figure}{section} 
 \numberwithin{table}{section} 
 
\newcommand{\mtx}[1]{\begin{bmatrix} #1 \end{bmatrix}}
\newcommand{\spec}{\operatorname{spec}}
\newcommand{\rank}{\operatorname{rank}}\newcommand{\tr}{\operatorname{tr}}\newcommand{\lam}{\lambda}

\newcommand{\proj}{\operatorname{proj}}

\newcommand{\Kn}{\operatorname{Kn}}
 
\newcommand{\bh}{{\bf h}}
\newcommand{\bv}{{\bf v}}
\newcommand{\bw}{{\bf w}}
\newcommand{\bz}{{\bf z}}
\newcommand{\bx}{{\bf x}}
\newcommand{\by}{{\bf y}}
\newcommand{\bzero}{{\bf 0}}
\newcommand\ones{{\mathbf 1}}

\newcommand{\R}{{\mathbb R}}

\newcommand\cE{{\mathcal E}}

\newcommand{\bit}{\begin{itemize}}
\newcommand{\eit}{\end{itemize}}
\newcommand{\ben}{\begin{enumerate}}
\newcommand{\een}{\end{enumerate}}
\newcommand{\beq}{\begin{equation}}
\newcommand{\eeq}{\end{equation}}
\newcommand{\bea}{\begin{eqnarray*}}
\newcommand{\eea}{\end{eqnarray*}}
\newcommand{\bpf}{\begin{proof}}
\newcommand{\epf}{\end{proof}}
\newcommand{\x}{\times}

\newcommand{\lp}{\!\left(}
\newcommand{\rp}{\right)}

\newcommand{\lb}{\left[}
\newcommand{\rb}{\right]}

\begin{document}

\bibliographystyle{plain}

\setcounter{page}{1}

\thispagestyle{empty}

 \title{\Title\thanks{Received
 by the editors on \DoS.
 Accepted for publication on \DoA. 
 Handling Editor: \HE. Corresponding Author: \CA}}

\author{Aida Abiad\thanks{Department of Mathematics and Computer Science, Eindhoven University of Technology, Eindhoven, The Netherlands (a.abiad.monge@tue.nl). Department of Mathematics: Analysis, Logic and Discrete Mathematics, Ghent University, Ghent, Belgium. Department of Mathematics and Data Science, Vrije Universiteit Brussel, Brussels, Belgium. Partially supported by the Research Foundation Flanders (FWO) grant 1285921N.} 
\and 
Leonardo de Lima\thanks{Departamento de Administração Geral e Aplicada,  Federal University of Parana, Curitiba, Brazil ({leonardo.delima@ufpr.br}).}
\and Dheer Noal Desai\thanks{Department of Mathematical Sciences, University of Delaware, DE, U.S.A. ({dheernsd@udel.edu}). Department of Mathematics and Statistics, University of Wyoming, WY, U.S.A. ({ddesai1@uwyo.edu}).}
\and Krystal Guo\thanks{Korteweg-de Vries Institute, University of Amsterdam, Amsterdam, The Netherlands ({k.guo@uva.nl}).}
\and Leslie Hogben\thanks{Department of Mathematics, Iowa State University,
Ames, IA, U.S.A. American Institute of Mathematics, San Jose, CA, U.S.A.
({hogben@aimath.org}).}
\and Jos\'e Madrid\thanks{Section de Mathematiques, Universite de Geneve, Geneva, Switzerland ({Jose.Madrid@unige.ch}).}
}

\markboth{\Names}{\Title}

\maketitle

\begin{abstract}
The energy of a graph $G$ is the sum of the absolute values of the eigenvalues of the adjacency matrix of $G$. Let $s^+(G), s^-(G)$ denote the sum of the squares of the positive and negative eigenvalues of $G$, respectively. It was conjectured by [Elphick, Farber, Goldberg, Wocjan,  \emph{Discrete Math.} (2016)] that if $G$ is a connected graph of order $n$, then $s^+(G)\geq n-1$ and $s^-(G) \geq n-1$. In this paper, we show partial results towards this conjecture. In particular,  numerous structural results that may help in proving the conjecture are  derived, including the effect of various graph operations. These are then used to establish the conjecture for several graph classes, including graphs with certain fraction of positive eigenvalues and unicyclic graphs.
\end{abstract}


\section{Introduction and preliminaries}

Suppose $G=( V(G),E(G))$ is a graph on $n=|V(G)|$ vertices with adjacency matrix  $A(G)$. Let $\pi_G$ be the number of positive eigenvalues and $\nu_G$ be the number of negative eigenvalues  and number the eigenvalues of  $A(G)$ in decreasing order, so $\lambda_1 \geq \cdots \geq \lambda_{\pi_G}>0$ are the positive eigenvalues  and $0 > \lambda_{n-\nu_G+1}\geq \cdots \geq \lambda_n$ are the negative eigenvalues. The \textsl{energy} $\cE(G)$ of $G$ is defined by
$\cE(G) = \sum_{i=1}^{\pi_G} \lambda_i -\sum_{j=n-\nu_G+1}^n \lambda_j$. Since its introduction  in the study of molecular chemistry more than sixty years ago, the energy of a graph has attracted a considerable amount of attention; for an overview see \cite{Gutman2001}.

In this paper, we investigate the positive and negative square energies of a graph, introduced by Wocjan and Elphick in \cite{WE13} to provide bounds on the chromatic number. The \textsl{square positive energy} of $G$ is defined to be $s^+(G) = \sum_{i=1}^{\pi_G} \lambda_i^2$. Similarly, the  \textsl{square negative energy} of $G$ is defined by $s^-(G) = \sum_{j=n-\nu_G+1}^n \lambda_j^2$.
The following conjecture is due to Elphick, Farber, Goldberg and Wocjan. 

\begin{conjecture}\label{conj:energies}{\rm \cite{EFGW16}} If $G$ is a connected graph of order $n$, then $s^+(G)\geq n-1$ and $s^-(G) \geq n-1$.
\end{conjecture}
We should note that in \cite{EFGW16}, the above conjecture is stated as a minimum, but we will consider it as two separate conjectures, namely that $s^+(G)\geq n-1$ and $s^-(G) \geq n-1$, which may be proved separately.
Conjecture \ref{conj:energies} has been established for several graph classes in \cite{EFGW16}{.} Using a bound on the chromatic number of a graph using $s^+$ and $s^-$ established by Ando and Lin in \cite{AndLin2015}, the conjecture was shown to be true for regular graphs (excluding odd cycles) in \cite{EFGW16}. Other classes for which the conjecture is known to be true include bipartite graphs, complete $q$-partite graphs, 
barbell graphs, hyper-energetic graphs, and graphs with exactly two negative eigenvalues  and minimum degree two \cite{EFGW16}. Moreover, a computational search among small graphs was done in \cite{EFGW16} and no counterexample was found.

 In this paper we present several results that support Conjecture \ref{conj:energies}. In Section \ref{sec:structural},  we provide some structural results, including establishing bounds on the effect of several graph operations and graph products.  We use a variety of  linear algebraic techniques to establish these results, including Perron-Frobenius theory, Rayleigh quotients, eigenvalue interlacing for edge deletion,  and  quotient matrices. 
 
 In Section \ref{sec:classesresults}, we apply the tools developed in Section \ref{sec:structural} to prove that the conjecture holds for various families of graphs, including odd cycles  and some other unicyclic graphs, some  cactus graphs, and extended barbell graphs.  We also find additional bounds. 
 Section \ref{s:conclude} contains concluding remarks and directions for future research.
 Preliminary results and additional background are stated in the remainder of this introduction.

We begin with  some graph notation and terminology that will be used throughout. 
 For a graph $G$,  $V(G)$ denotes the set of vertices  and  $E(G)$ denotes the set of edges.  Recall that a graph $H$ is a \emph{subgraph} of a graph $G$ if $V(H)\subseteq V(G)$ and $E(H)\subseteq E(G)$ and $H$ is an 
 \emph{induced subgraph} of a graph $G$ if $V(H)\subseteq V(G)$ and $E(H)=\{uv\in  E(G): u,v\in V(H)\}$.  Two vertices $u$ and $v$ of $G$ are {\em neighbors} (or $u$ and $v$ are \emph{adjacent}) if $uv\in E(G)$.  The (open) \emph{neighborhood} of a vertex $v$, denoted by $N_G(v)$ or $N(v)$ is the set of neighbors of $v$, and the \emph{closed neighborhood} is $N_G[v]=N_G(v)\cup\{v\}$.   The \emph{degree} of $v\in V(G)$ is $\deg(v)=|N(v)|$. 
 
  Let $\bar d(G)$ denote the average degree of $G$. Note $\bar d(G)=\frac{2}n |E(G)|$ where $n$ is the order of $G$.

 As noted in \cite{EFGW16}, there is a slightly stronger version of Conjecture \ref{conj:energies}: $s^+(G), s^-(G) \geq n-k$ for a graph $G$ of order $n$ with $k$ connected components. However, in this paper we will focus on connected graphs, because it is shown in \cite{EFGW16} that Conjecture \ref{conj:energies} implies $s^+(G), s^-(G) \geq n-k$ when $G$ has $k$ connected components. 
Suppose $G$ is a connected graph of order $n$. Since the sum of the squares of all of the eigenvalues is equal to $2|E(G)|$  and $|E(G)|\ge n-1$, it immediately follows that
\[
 \max(s^+(G), s^-(G)) \geq n-1. 
 \]
If $G$ is bipartite, then $s^+(G) = s^-(G) = |E(G)| \geq n-1$, with equality if and only if $G$ is a tree.

We observe that the proof  of the conjecture for regular graphs (excluding odd cycles) in \cite{EFGW16}
 gives $s^+(G),s^-(G)\ge \frac{2|E(G)|}{\chi(G)}$; loosely speaking, this  is useful for a
 graph with many edges and low chromatic number.  
 By the four color theorem, the conjecture is true for planar graphs with at least $2n-2$ edges, which includes all maximal planar graphs when $n\ge 4$.  

 Let $\rho(A)$ denote the spectral radius of the matrix $A$ and $\rho(G)=\rho(A(G))$. Note that $\rho(G)=\lambda_1(G)$ by Perron-Frobenius theory since $A(G)$ is nonnegative. 
\begin{remark}{\rm
 Using Rayleigh quotients, it is easy to see that $\rho(G)=\lambda_1(G)\ge \bar d(G)$, so $s^+(G)\ge \bar d(G)^2$. Thus if $\bar d(G)^2\ge n-1$, then $s^+(G)\ge n-1$.   }\end{remark}

\begin{remark}{\rm \label{r:rho+uv} 
It is known  that $\rho(G+uv)\ge \rho(G)$, which follows from Rayleigh quotients for non-negative matrices. So if $G$ has a spanning subgraph $H$ for which $\rho(H)^2\ge n-1$, then $s^+(G)\ge \rho(G)^2\ge \rho(H)^2\ge n-1$.  In particular,
\begin{enumerate}
\item If $G$ has a dominating vertex (a vertex adjacent to every other vertex), then $s^+(G)\ge n-1$, because $\rho(K_{1,n-1})=\sqrt{n-1}$.
\item\label{compbip} If $G$ has a $K_{ r,n-r}$ subgraph for $1\le r \le {n-1}$, then $s^+(G)\ge n-1$, because every bipartite graph satisfies Conjecture \ref{conj:energies} (see \cite{EFGW16}) and a complete bipartite graph has exactly one positive eigenvalue (the spectral radius $\rho(K_{r,n-r})$) and exactly one negative eigenvalue ($-\rho(K_{r,n-r})$).  
\item If $G$ has a clique {$K_{r+1}$} with $r\ge \sqrt{n-1}$, then $s^+(G)\ge n-1$, because $\rho(K_{r+1})=r\ge \sqrt{n-1}$.
\end{enumerate}}
\end{remark}

\section{Structural techniques and tools}\label{sec:structural}

 In this section we establish bounds on the changes in $s^+(G)$ and $s^-(G)$ caused by  graph operations and products and apply quotient matrices to the study of $s^+(G)$ and $s^-(G)$.  In particular, in Section  \ref{s:Kron}, we show that if $G$ and $H$ satisfy Conjecture \ref{conj:energies}, then so does $G\otimes H$.  In Section \ref{s:edge}, we bound the changes in $s^+(G)$ and $s^-(G)$ caused by removing an edge. In  Section \ref{subsec:movingneighbors}, we bound the changes in $s^+(G)$ and $s^-(G)$ caused by moving neighbors from one vertex to another. 
 {Sections \ref{sec:interlacing} and \ref{sec:haemersinterlacing} apply interlacing to induced subgraphs,
and to vertex partitions via
 quotient matrices. Section \ref{sec:eqot-twins} uses quotients of equitable partitions to determine spectra of  graphs with twins.}

We begin with a simple result for joins and its application to threshold graphs.
For  disjoint graphs $G$ and $H$,  the \emph{join} of $G$ and $H$, denoted by $G\vee H$, is the graph
with vertex set $V(G) \cup V(H)$ and edge set $E(G) \cup E(H) \cup \{ gh : g \in V(G), h\in V(H)\}$.
A \emph{threshold graph} is a graph that can be constructed from a one-vertex graph by repeated applications of the following two operations:
1) addition of a single isolated vertex to the graph; 2) addition of a single dominating vertex to the graph, i.e., a single vertex that is connected to all other vertices.

\begin{lemma}
\label{lema:join}
Let $G,H$ be two non-empty  disjoint graphs with $|V(G)|+|V(H)|=n$. Then,\vspace{-6pt} \[s^+(G\vee H)\ge n-1.\vspace{-6pt} \]
\end{lemma}
\begin{proof}
Note that $G\vee H$ contains $K_{s,n-s}$ as a subgraph for some $1\le s\le n-1$,  so  $s^+(G\vee H)\ge  n-1$ by Remark \ref{r:rho+uv}. 
\end{proof}  

The method in Lemma \ref{lema:join} does not resolve the conjecture for $s^-(G\vee H)$.  Note that there is a tight example from taking $G=K_1$ and $H=K_{n-1}$ shows that we can have $s^-(G\vee H)=n-1$.

The following result is immediate from Lemma \ref{lema:join} and the definition of threshold graph.    

\begin{corollary}
If $G$ is a  connected threshold graph {of order $n$}, then $s^+(G)\ge n-1$.
\end{corollary}

\subsection{Kronecker product}\label{s:Kron}
The \textsl{Kronecker product} (also called the \textsl{tensor product}) of two graphs $G$ and $H$, denoted $G\otimes H$, is the graph with vertex set $V(G) \times V(H)$
where $(g,h)$ and $(g',h')$ are adjacent whenever $g \sim g'$ in $G$ and $h \sim h'$ in $H$. The adjacency matrix of $G\otimes H$ is $A(G) \otimes A(H)$, where $\otimes$ denotes the Kronecker product of matrices.

\begin{proposition}\label{lem:productsgeneral}
Suppose $G$ and $H$ are 
graphs of orders $n$ and $m$, respectively, that satisfy $s^+(G)$, $s^-(G)\ge n-1$, $s^+(H),s^-(H)\ge m-1$, and $m,n\ge 3$.  Then  $s^+(G\otimes H),s^-(G\otimes H)\ge mn-1$.
\end{proposition}
\begin{proof}
  Let $\lambda_i,\mu_j$ denote the eigenvalues of $G,H$ with $\lambda_{\pi_G}> 0,\lambda_{\pi_G+1}\le 0$ and $\mu_{\pi_H}> 0,\ \mu_{\pi_H+1}\le 0$.
The eigenvalues of $G\otimes H$ are $\lambda_i \mu_j$.  Thus

\bea
s^+(G\otimes H) &=& \sum_{ i\le \pi_G,j\le \pi_H} \lambda_i^2\mu_j^2+\sum_{ i>\pi_G,j> \pi_H} \lambda_i^2\mu_j^2\\
 &=&  s^+(G)s^+(H)+s^-(G)s^-(H) \\
&\ge& 2(n-1)(m-1)\\
&=&2nm-2m-2n+2\ge nm-1,
\eea
where in the last step we used that $nm+1\ge2(m+n-1)$, which is equivalent to $(m-2)(n-2)\geq 1$ and which is valid for $m,n\ge 3$.
A similar computation shows that the same holds for $s^-(G\otimes H)$.
\end{proof}

Note that  the conclusion of Proposition  \ref{lem:productsgeneral} can be false if, for example, $G=K_2$ and  $H$ is a tree (so $H$ satisfies the conjecture with equality) because $G\otimes H$ is two copies of $H$.  However,  this is not a counterexample to the more general conjecture  (since $G\otimes H$ has 2 components and does satisfy $s^+(G\otimes H),s^-(G\otimes H)=  mn-2$) nor to the proposition (since we require $m, n \ge 3$).

We note that this naive method did not work for the Cartesian, categorical or strong products of graphs.

\subsection{Removing an edge}\label{s:edge}
In this section we use a result on edge interlacing for the adjacency matrix established by Hall, Patel, and Stewart in  \cite{edgeinterlacing} to investigate what we can say regarding the conjecture when we consider the subgraph obtained on deleting an edge in the original graph.

\begin{theorem}{\rm \cite{edgeinterlacing}}\label{thm:edgeinterlacing}
Let $G$ be a graph and let $H = G - e$, where $e$ is an edge of $G$. If $\lambda_1 \geq \cdots \geq \lambda_n $ and $\theta_1\geq \cdots \geq \theta_{n}$ denote the eigenvalues of $A(G)$ and $A(H)$, respectively, then \vspace{-5pt}
\[\lambda_{i-1} \geq \theta_i \geq \lambda_{i+1} \qquad \text{for } i=2,3,\ldots, n-1, \vspace{-5pt}\]
and $\theta_1\geq \lambda_2$ and $\theta_{n}\leq \lambda_{n-1}$.

\end{theorem}

We use Theorem \ref{thm:edgeinterlacing}  to obtain the next result, which gives lower bounds for $s^+(G)$ and $s^-(G)$ in terms of $s^+(G-e)$ and $s^-(G-e)$. The slightly  stronger bound for $s^+(G)$ is obtained by further combining Theorem \ref{thm:edgeinterlacing} with the known fact that the spectral radius of $G$ is at least as much as the spectral radius of $G - e$. 
 Recall that it is known that Conjecture \ref{conj:energies} is true for a graph that has exactly one positive eigenvalue or  exactly one negative eigenvalue 
 \cite{EFGW16}, so the restriction to having at least two positive and two negative eigenvalues does not reduce the usefulness of the result.

\begin{theorem}
 Let $G$ be a graph,  let $H = G - e$ (where $e$ is an edge of $G$), and let $\lambda_1 \geq \cdots \geq \lambda_n $ and $\theta_1\geq \cdots \geq \theta_{n}$ denote the eigenvalues of $A(G)$ and $A(H)$, respectively. If $H$ has at least two positive eigenvalues  and at least two negative eigenvalues, then $s^+(G)\geq s^+(H) - \theta_{2}^2$ and $s^-(G)\geq s^-(H) - \theta_{n}^2$.
\end{theorem}
\begin{proof}
Let $\pi_H\geq 2$ be the number of positive eigenvalues of $H$, and $\nu_H\ge 2$ be  the number of negative eigenvalues of $H$. 
Theorem \ref{thm:edgeinterlacing} gives us that
\[
\lambda_{i-1} \geq \theta_{i} \geq \lambda_{i+1} \qquad \text{for } i=2,\ldots,n-1,
\]
$\theta_1\geq \lambda_2$ and $\theta_{n}\leq \lambda_{n-1}$. Further, we may also observe that $\lambda_1 \ge \theta_1$ since $H \subset G$.

We thus have that 
\[
\lambda_{i-1}^2 \geq \theta_i^2, \quad \text{for } i = 2, \ldots, \pi_H
\]

and $\lambda_1^2 \geq \theta_1^2\geq \lambda_2^2$.

Analogously, we have that
\[
\lambda_{i+1}^2 \geq \theta_{i}^2, \quad \text{for } i = { n -\nu_H+1} 
, \dots,n-1
\]

and $\theta_{n}^2\geq \lambda_{n-1}^2$.

Thus, we have that 
\[
s^+(G) \geq \sum_{i=2}^{\pi_H} \lambda_{i-1}^2 = \lambda_1^2 + \sum_{i=3}^{\pi_H} \lambda_{i-1}^2 \geq \theta_1^2 + \sum_{i=3}^{\pi_H}\theta_i^2  = s^+(H) - \theta_{2}^2
\]
and 
\[
s^-(G) \geq \sum_{i={ n}-\nu_H+1}^{n-1} \lambda_{i + 1}^2 \geq \sum_{i={n}-\nu_H+1}^{n-1} \theta_i^2 = s^-(H) - \theta_{n}^2,
\]
as claimed.
\end{proof}

We obtain the following result as an immediate corollary.
\begin{corollary}
If $G$ is a graph on $n$ vertices and $H = G - e$ (where $e$ is an edge of $G$) satisfies $s^+(H) - \theta_{2}^2\ge n-1$ and $s^-(H) - \theta_{n}^2\ge n-1$, then  $s^+(G),s^-(G) \geq n-1$. 
\end{corollary}

\subsection{Moving neighbors from one vertex to another}\label{subsec:movingneighbors}

In this section we will focus on the following operation on graphs discussed in \cite{WSL2009}. For a graph {$G$} 
with two vertices $u$ and $v$ {and a set of vertices} $\{w_1, w_2, \ldots w_r\} \subseteq N_G(v) \setminus (N_G(u) \cup \{u\})$, let {$G_{u, v}$} 
denote the graph with vertex set $V(G_{u,v}) = V(G)$ and edge set $ E(G_{u, v}) = (E(G) \setminus \{v w_i \textrm{ for } 1 \le i \le r\}) \cup \{u w_i \textrm{ for } 1 \le i \le r\}.$
We say that $G_{u, v}$ is the graph obtained by \textit{moving the neighbors $ \{w_1, w_2, \ldots, w_r\}$ of $v$ to $u$} (note that defining $G_{u,v}$ as moving the neighbors $ \{w_1, w_2, \ldots, w_r\}$ of $v$ to $u$ implies $ \{w_1, w_2, \ldots, w_r\}$ satisfy the  condition of the definition). An example of this operation is shown in Figure \ref{fig:movingnbrs}. {Note that the symbol $G_{u, v}$ can denote more than one graph, since the set of vertices  moved is not embedded in the notation.}

\begin{figure}[htbp]
    \centering
    \includegraphics[scale=0.5]{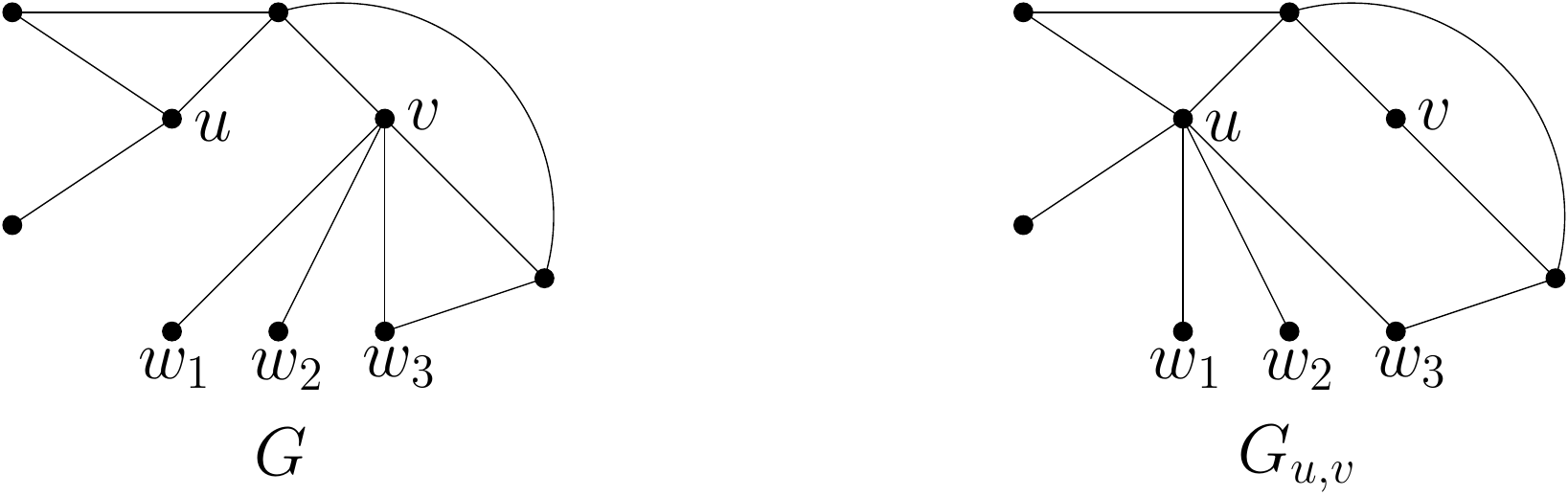}
    \caption{Graph $G_{u,v}$ is obtained from $G$ by moving the neighbors $ \{w_1,w_2,w_3\}$ of $v$ to $u$. 
    \label{fig:movingnbrs}}
\end{figure}

{Wu, Shao, and Liu  proved the next  result describing  interlacing  of the spectra of $G$ and $G_{u, v}$ in \cite{WSL2009}. }

{\begin{theorem}{\rm \cite{WSL2009}}
\label{moving neighbors interlacing lemma}
Let $G$ be a graph, let $u,v$ be two vertices of $G$, and let 
$G_{u, v}$ be the graph obtained from $G$ by moving the neighbors $\{w_1, w_2, \ldots, w_r \}$ of $v$ to $u$. 
Denote the eigenvalues of $A(G)$ and $A(G_{u,v})$ by $\lambda_1 \geq \cdots \geq \lambda_n $ and $\theta_1\geq \cdots \geq \theta_{n}$, respectively.  Then
\[\lambda_{i-1} \geq \theta_i \geq \lambda_{i+1}, \textrm{ for } i=2,3,\ldots, n-1,\ \theta_{1} \geq \lambda_2, \textrm{ and } \lambda_{n-1}\,  \geq \, \theta_n.\]  
\end{theorem}}

{In the next result we apply the preceding theorem to two special cases to improve the bound on the spectral radius.}  
\begin{lemma}
\label{vertex identification spectral radius lemma}
Let {$G$} 
be a graph with spectral radius {$\rho(G)$} 
and Perron vector $\bx=[x_i]$ {and let $G_{u,v}$ be the graph obtained by  moving $\{w_1, w_2, \ldots, w_r\}$ from $v$ to $u$. 
If} 
\begin{enumerate}
    \item ${x}_u \ge {x}_v$, or
    \item $N_G(v) \cap (N_G(u) \cup \{u\}) = \emptyset$ and  $\{w_1, w_2, \ldots, w_r\} = N_G(v)$,
\end{enumerate}
then 
{$\rho(G_{u,v}) \ge \rho(G)$.} 
Furthermore, if $G$ is connected, then {$\rho(G_{u,v}) > \rho(G)$.} 
\end{lemma}

\begin{proof}
{Let $\theta_1=\rho(G_{u,v})$ and $\lambda_1=\rho(G)$.}
For a graph $H = (V(H), E(H))$ with spectral radius {$\rho(H)$}, and Perron vector $\bh$, the Rayleigh {quotient} characterization of the spectral radius gives \[\rho(H) = \max_{\bz \neq \bzero} \dfrac{\sum_{ij \in E(H)} 2{z_i z_j}}{\sum_{i \in V(H)}z_i^2} = \dfrac{\sum_{ij \in E(H)} 2h_i h_j}{\sum_{i \in V(H)}h_i^2},\] where $\bz$ runs over all non-zero vectors.
To prove the result, {we show that with either of the hypotheses,} \[\theta_1 - \lambda_1 =  \left( \max_{ \bz \neq \bzero} \dfrac{\sum_{ij \in E(G_{u, v})} 2z_i z_j}{\sum_{i \in V(G_{u, v})}z_i^2}\right) - \dfrac{\sum_{ij \in E(G)} 2{x}_i {x}_j}{\sum_{i \in V(G)}{x}_i^2} \ge 0\]
{with strict inequality if $G$ is connected.}
Recall that $V(G_{u, v}) = V(G)$ and all the entries of the Perron vector of any graph are nonnegative. 

{We consider the two cases separately.}

\begin{description}
    \item[Case 1.]Let $\mathrm{x}_u \ge \mathrm{x}_v$.
Consider the vector $\bz = \bx$. Then, 
\beq\label{eq:nhbrmov}
\theta_1 - \lambda_1 \geq  \dfrac{\sum_{ij \in E(G_{u, v})} 2x_i x_j - \sum_{ij \in E(G)} 2x_i x_j}{\sum_{i \in V(G)}x_i^2} = \dfrac{ \sum_{i = 1}^r 2 \, (x_u - x_v) x_{ w_i}}{\sum_{i \in {V(G)}}x_i^2} \ge 0.
\eeq
{Now suppose $G$ is connected,  which implies the Perron vector $\bx$ is a  positive vector. Assume that $\theta_1 = \lambda_1$. 
Then all the inequalities must be equalities in equation \eqref{eq:nhbrmov}. Requiring that the first inequality be equality implies that  \[\frac{\bx^T(A(G_{u,v})\bx}{\bx^T\bx}=\max_{\bz\ne\bzero}\frac{\bz^T(A(G_{u,v})\bz}{\bz^T\bz}.\] Thus $\bx$ is also a Perron vector for $G_{u, v}$. } Then, 
\[\sum_{ y \sim_G u} {x}_y = (A(G)\bx)_u=\lambda_1 x_u=\theta_1 x_u=  (A(G_{u,v})\bx)_u=\sum_{ y \sim_G u} {x}_y + \sum_{i=1}^r{x}_{ w_i},\]  where the equality $\theta_1 x_u  = (A(G_{u,v})\bx)_u$ appears because $\bx$ is a Perron vector of $A(G_{u,v})$. 
However, this is a contradiction since  $\sum_{i=1}^r {x}_{ w_i} > 0$. Thus, $\theta_1 > \lambda_1$ when $G$ is connected.
\item[Case 2.] Let $N_G(v) \cap (N_G(u) \cup \{u\}) = \emptyset$. If $x_u \ge x_v$, then we are done using the previous case. So we assume $x_v > x_u$. Consider the vector $\bz$ defined as follows:
\[{z}_i = \begin{cases}
{x_i}, \textrm{ for } i \neq u, v,\\
{x_v}, \textrm{ for } i = u,\\
0, \textrm{ for } i = v.
\end{cases}\]

Then 
\[\theta_1 - \lambda_1 \geq  \dfrac{\sum_{ij \in E(G_{u, v})} 2{z}_i {z}_j}{(\sum_{i \in V(G)}{x}_i^2) - {x}_u^2} - \dfrac{\sum_{ij \in E(G)} 2{x}_i {x}_j}{\sum_{i \in V(G)}{x}_i^2} \ge \dfrac{\sum_{i=1}^r 2({x}_v - {x}_v) {x}_{ w_i}}{\sum_{i \in V(G)}{x}_i^2} = 0.
\]
Finally, we note that if $G$ is connected, then ${x_u} > 0$ and 
 therefore the second inequality in the preceding equation is strict, so  $\theta_1 > \lambda_1$.
\end{description}\end{proof}

 Now we are ready to establish the next result, which gives a lower bound for $s^+(G_{u, v})$ in terms of $s^+(G)$  and $s^-(G_{u, v})$ in terms of $s^-(G)$.

\begin{theorem}
\label{moving neighbors squared energies bound}
Let {$G$} be a graph,  let $G_{u,v}$ be the graph obtained by moving the neighbors {$\{w_1, w_2, \ldots, w_r\} $} from $v$ to $u$, and let $\lambda_1 \geq \cdots \geq \lambda_n $ and $\theta_1\geq \cdots \geq \theta_{n}$ denote the eigenvalues of $A(G)$ and $A(G_{u,v})$, respectively. 
Then $s^+(G_{u,v}) \geq s^+(G) - \lambda_1^2$ 
and 
$s^-(G_{u,v})  \geq s^-(G)-\lambda_n^2$.

If 
$1)$ ${x}_u \ge {x}_v$ (where $\bx$ is the Perron vector of $A(G)$) or
$2)$  $N(u) \cap (N(v) \cup \{v\})=\emptyset$ and  $ \{w_1, w_2, \ldots, w_r\} = N_G(v)$,
 then the bound for $s^+(G)$ can be improved to
$s^+(G_{u,v}) \geq s^+(G) - \lambda_2^2$.
\end{theorem}
\begin{proof} Let $\pi_G$ and $\pi_{G_{u,v}}$ denote the number of positive eigenvalues of $G$ and $G_{u,v}$ respectively.
 Theorem \ref{moving neighbors interlacing lemma}  implies that
\[
 \theta_{i} \geq \lambda_{i+1} \mbox{ for } i=1,\ldots,\pi_{G_{u,v}},
\qquad\mbox{ and }\qquad
\lambda_{i} \geq \theta_{i+1 }  \mbox{ for } i=\pi_{G_{u,v}}+1,\ldots,n-1.
\]
Observe that $\lambda_{\pi_{G}-1} \geq \theta_{\pi_{G}} > 0$ and $0 \ge \theta_{\pi_{G_{u, v}} + 1} \ge \lambda_{\pi_{G_{u, v}} + 2}$, so $\pi_{G_{u, v}}-1 \leq \pi_G \leq \pi_{G_{u, v}} + 1.$
We break the proof of $s^+(G_{u, v}) \ge s^+(G) - \lambda_1^2$ into two parts. 
If $\pi_G = \pi_{G_{u, v}} + 1$, then 
 $
\theta_{i}^2 \geq \lambda_{i+1}^2$ for $ i = 1,\dots,\pi_{G_{u, v}}, 
$
and
\[
s^+(G_{u, v}) = \sum_{i=1}^{\pi_{G_{u,v}}} \theta_i^2 \geq  \sum_{i=2}^{\pi_{G}}\lambda_{i}^2 = s^+(G) - \lambda_1^2.
\]
If instead $\pi_G \le \pi_{G_{u, v}}$, then
$\theta_{i}^2 \geq \lambda_{i+1}^2$ for $ i = 1,\dots,\pi_{G} -1$ and
 \[
s^+(G_{u, v}) =  \left(\sum_{i=1}^{\pi_{G_{u,v}} - 1} \theta_i^2\right) + \theta_{\pi_{G_{u,v}}}^2 
\geq  \sum_{i=2}^{\pi_{G}}\lambda_{i}^2 + \theta_{\pi_{G_{u,v}}}^2 \ge s^+(G) - \lambda_1^2.
\]
 The proof that $s^-(G_{u, v}) \ge s^-(G) - \lambda_n^2$ is similar.

Now assume that $G$, $u$,  and $v$ satisfy the hypotheses of Lemma \ref{vertex identification spectral radius lemma}.  Then $ \sum_{i=1}^{p} \theta_i^2 $ and $ \sum_{i=2}^{\pi_G} \lambda_i^2 $ (where $p$ equals $\pi_{G_{u,v}}$ or  $\pi_{G_{u,v}}-1$ as needed) can be replaced by $ \theta^2_1+ \left(\sum_{i=2}^{p} \theta_i^2\right) $ and $\lambda_1^2+ \left(\sum_{i=3}^{\pi_G} \lambda_i^2\right) $ to obtain $s^+(G_{u,v}) \geq s^+(G) - \lambda_2^2$.
\end{proof}

\begin{remark} {\rm The process of moving neighbors from $v$ to $u$ is reversible.  That is, $\lp G_{u,v}\rp_{v,u}=G$ when the same set of vertices $\{w_1,\dots,w_r\}$ is used for each  move.  Thus from Theorem \ref{moving neighbors squared energies bound} we also obtain the  bounds $s^{+}(G) \geq s^+(G_{u,v})-\theta_{1}^2$ and $s^{-}(G) \geq s^-(G_{u,v})-\theta_{n}^2$.}  \end{remark}

\subsection{Induced subgraphs}\label{sec:interlacing}

{We  use interlacing to obtain lower bounds for the squared energies of $G$ in terms of the squared energies of induced subgraphs of $G$.} 

\begin{lemma}\label{lemma:induced subgraph}
Suppose a connected graph $G$ 
has an induced subgraph $H$.  
Then  $\pi_G\ge \pi_H$, $\nu_G\ge \nu_H$,   $s^+(G) \geq s^+(H)$, and $s^-(G) \geq s^-(H)$. \end{lemma}
\begin{proof}
Let $n_G$  be the number of vertices of $G$  and let $n_H$  be the number of vertices of $H$. 

Here we denote the $i$th largest {eigenvalues of $G$ and $H$ by $\lambda_i(G)$ and $\lambda_i(H)$}, so
$\lambda_i(G) \geq \lambda_i(H) \geq \lambda_{n_G-n_H+i}({G})$ by the Interlacing Theorem.   This implies that $\pi_G\ge \pi_H$ and $\nu_G\ge \nu_H$. Furthermore,
\[
\lambda_i(G)^2 \geq \lambda_i(H)^2, \quad i = 1, \ldots, \pi_H
\ \mbox{ 
and } \ 
\lambda_{n_G-n_H+i}(G)^2 \geq \lambda_i(H)^2, \quad i = {n_H}-\nu_H+1,\dots,{n_H}.
\]
{This implies $s^+(G)\ge s^+(H)$ and $s^-(G)\ge s^-(H)$.}
\end{proof}

We obtain the next result as an immediate corollary. 
\begin{corollary}
If a graph $G$ on $n$ vertices has an induced subgraph $H$ with $s^+(H)\ge n-1$ {(respectively, $s^-(H)\ge n-1$)}, then $s^+(G)\ge n-1$ {(respectively, $s^-(G) \geq n-1$)}. 
\end{corollary}

Since we may not know information about the squared energies of the induced subgraphs,  the next result may be more useful (it is immediate from the fact that a bipartite graph $H$ with $\ell$ edges has $s^+(H) = s^-(H)= \ell$).

\begin{corollary}\label{cor:induced subgraph}
If a graph $G$ on $n$ vertices has an induced bipartite subgraph with at least $n-1$ edges, then $s^+(G)\ge n-1$ and $s^-(G) \geq n-1$. 
\end{corollary}

Corollary \ref{cor:induced subgraph} is applied to cactus graphs in Section \ref{s:cacti}.

\subsection{Quotient  matrices}\label{sec:haemersinterlacing}

Let $M=[m_{ij}]$ be an $n\x n$ matrix and  $X=(X_1,\dots,  X_p)$ be a partition of {$ \{1,\dots,n\}$}. 
 Then the partition $X$ defines a {$p\x p$ 
 matrix} $[M_{ij}]$ where $M_{ij}=M[X_i|X_j]$ {is the submatrix with row indices in $X_i$ and column indices in $X_j$.
 
 The {\em characteristic matrix}  $S=[s_{ij}]$  of $X$ is the $n\x p$ matrix defined by  $s_{ij}=1$ if $i\in X_j$ and $s_{ij}=0$ if $i\not\in X_j$.

The {\em quotient matrix} $B=[b_{ij}]$ of $M$ for this partition is the $p\x p$ matrix with entry $b_{ij}$ equal to the average row sum of the submatrix $M_{ij}$.  More precisely,
$b_{ij}= \frac{1}{|X_{i}|}\ones_n^{T}M_{ij}\ones_p=\frac{1}{|X_{i}|}(S^{T}MS)_{ij}$
where $\ones_k$ is a $k$-vector with every entry equal to one.

{\begin{lemma}{\rm \cite
{H1995}}\label{cor:2.3Haemerspaper} If $B$ is the quotient matrix of a {symmetric  matrix $M$ with respect to a partition}, then the eigenvalues of $B$ interlace the eigenvalues of $M$.\end{lemma}
 The proof of the next result is analogous to the proof of Lemma \ref{lemma:induced subgraph} using the previous result.
\begin{proposition}
\label{lemma:quotient}
If $G$ has a partition of the vertices  with quotient matrix $B$ of $A(G)$, then $s^+(G) \geq s^+(B)$ and $s^-(G) \geq s^-(B)$. 
\end{proposition}}

 {We can give a simple application of Proposition \ref{lemma:quotient} for a graph with an edge cut. 

\begin{lemma}\label{l:partition-quotient}
    Let $G$ be a graph on $n$ vertices and let $S$ be a subset of $ V(G)$ of order $s$. Let $c$ be the number of edges that are incident to exactly one vertex of $S$. Let $d_1$ be the average degree of the subgraph of $G$ induced by $S$ and $d_2$ be the  average degree of the subgraph of $G$ induced by ${V(G)} \setminus S$. If $d_1d_2 - \frac{c^2}{(n-s)s} \geq 0$, then
    \[ s^+(G) \geq   d_1^2 + d_2^2 + \frac{{2}c^2}{s(n-s) }.
    \]
If $d_1d_2 - \frac{c^2}{(n-s)s} < 0$, then
   {\scriptsize \[
    s^+(G) \geq \lambda_1(G)^2\ge  \frac{1}{4} \left(d_1 + d_2 + \sqrt{{ (d_1 -d_2)^2 +}\frac{4c^2}{s(n-s) }} \right)^2
     \text{ and  }
  s^-(G) \geq \lambda_n(G)^2\ge  \frac{1}{4} \left(d_1 + d_2 { -} \sqrt{{ (d_1 -d_2)^2+} \frac{4c^2}{s(n-s) }} \right)^2 .
    \]}
\end{lemma}

\begin{proof} We consider the partition of $V(G)$ into $S$ and $V\setminus S$. The quotient matrix of $A(G)$ with respect to this partition is 
\[
B = \mtx{d_1 & \frac{c}{s} \\ \frac{c}{n-s} & d_2}.
\]
We can find the characteristic polynomial of $B$ in variable $t$  as follows:
\[
|tI_2 - B| = \begin{vmatrix} t - d_1 & -\frac{c}{s} \\ -\frac{c}{n-s} & t- d_2 \end{vmatrix} = (t - d_1)(t- d_2) - \frac{c^2}{s(n-s)}
= t^2 - (d_1 + d_2) t + d_1d_2 - \frac{c^2}{s(n-s)}.
\]
Thus the eigenvalues of $B$ are 
\[
\lambda_{\pm} = \frac{d_1 + d_2 \pm \sqrt{(d_1 + d_2) ^2 - 4d_1d_2  + \frac{4c^2}{s(n-s)} } }{2} = \frac{d_1 + d_2 \pm \sqrt{ (d_1 -d_2)^2 +\frac{4c^2}{s(n-s)} } }{2}.
\]
{
If the determinant $d_1d_2 - \frac{c^2}{(n-s)s} \geq 0$ then  $\lambda_{-} > 0$ and $s^+(G) \geq \lambda_{+}^2 + \lambda_{-}^2$ and thus 
\bea
   s^+(G) &\geq &\frac{1}{4} \left(d_1 + d_2 + \sqrt{ (d_1 -d_2)^2 +\frac{4c^2}{s(n-s) }} \right)^2 
 + \frac{1}{4} \left(d_1 + d_2 -\sqrt{ (d_1 -d_2)^2+\frac{4c^2}{s(n-s) } } \right)^2 \\
 &=&  \frac{1}{2} \left(d_1 + d_2 \right)^2 + \frac{1}{2} (d_1 -d_2)^2+ \frac{2c^2}{s(n-s)}  \\
 &=& d_1^2 + d_2^2 + \frac{2c^2}{s(n-s). }
\eea
 Otherwise, $\lambda_{-}< 0$ and we obtain 
$s^+(G) \geq  \lambda_1(G)^2\ge  \lambda_{+}^2$ and $s^-(G) \geq  \lambda_n(G)^2\ge  \lambda_{-}^2$ from Proposition \ref{lemma:quotient}.}
\end{proof}

 Lemma \ref{l:partition-quotient} will be applied in Section \ref{subsec:unicyclic} to establish the conjecture for a family of unicyclic graphs.  Here we apply Lemma \ref{l:partition-quotient} to conclude that a join of a graph $H$ with itself  satisfies the conjecture provided $H$ does not have too high density.
\begin{proposition}
     Suppose $H$ is a graph of order $r\ge 8$ with average degree
$d \le \frac r 2$.  Then $G=H\vee H$ satisfies Conjecture \ref{conj:energies}.
\end{proposition}
\bpf Note that the order $n$ of $G$ is $2r$.  Since $G$ has a $K_{r,r}$ subgraph, $s^+(G)\ge n-1$.  Partition the vertices of $G$ into the vertices of the two copies of $H$.  Then $d_1=d_2=d, s=n-s=r,$ and $c=r^2$.  So $d_1d_2 - \frac{c^2}{(n-s)s} \le \frac{r^2}4-r^2< 0$ and by Lemma \ref{l:partition-quotient},
\[s^-(G) \geq \frac{1}{4} \left(2d { -} \sqrt{ 4r^2} \right)^2\ge \lp \frac r 2-r \rp^2=\frac{r^2}4\ge 2r-1=n-1,
    \]
where the last inequality is holds because  $r\ge 8$.
\epf
Note that $d_1d_2 - \frac{c^2}{(n-s)s}<0$ without the assumption that $d\le \frac r 2$, since $d\le r-1$ for any graph of order $r$.

\subsection{Equitable partitions and twins}\label{sec:eqot-twins}
{Let $M=[m_{ij}]$ be an $n\x n$ matrix.  The partition $X=(X_1,\dots,  X_p)$  of $ \{1,\dots,n\}$ is {\em equitable}} for   $M$ if for every pair $i,j\in\{1,\dots,p\}$, the row sums of $M_{ij}$ are constant.

  {The  material  in this section will be applied in Section \ref{sec:extendedbarbell}. It} is adapted from \cite{HR21}, where analogous results for distance matrices are presented; the results there could also be adapted to show similar results for the Laplacian, signless Laplacian, and normalized Laplacian matrices of a graph.

\begin{lemma}\label{t:quot}
Let $M$ be a symmetric $n\x n$  matrix, let $X$ be an equitable partition of {$\{1,\dots,n\}$},  let $B$ be the quotient matrix  of $M$ for  $X$, and let $\bw,\by,\bz\in \R^p$.
\ben[$(1)$]
\item\label{l:quot-0} {\rm \cite[p. 24]{BH}} $MS=SB$.
\item\label{l:quot-1}{\rm \cite{HR21}} If $i\in X_j$, then $(S\bw)_i={w}_j$ where $(S\bw)_i$ denotes the $i$th coordinate of $S\bw$ and $w_j$ denotes the $j$th coordinate of $\bw$.
\item\label{l:quot-2}{\rm \cite{HR21}} If $S\bw=S\by$, then $\bw=\by$.
\item\label{l:quot-3}{\rm \cite{HR21}} If $S\bz$ is an eigenvector of $M$, then  $\bz$ is an eigenvector of $B$ for the same eigenvalue.
\item\label{l:quot-4} {\rm \cite[Lemmas 2.3.1]{BH}} If $\bz$ is an eigenvector of $B$, then  $S\bz$ is an eigenvector of $M$ for the same eigenvalue. 
\een\end{lemma}

 Let $v_1, v_2$ be vertices of a graph $G$ of order at least three  that  have the same neighbors other than $v_1$ and $v_2$.
  If $N(v_1)=N(v_2)$ (so $v_1$ and $v_2$ are not adjacent), then $v_1$ and $v_2$ are called {\em independent twins}. If $N[v_1]=N[v_2]$ (so $v_1$ and $v_2$ are adjacent), then $v_1$ and $v_2$ are   called {\em adjacent twins}. Both cases are referred to as {\em twins}.  Note that twins have the same degree.  Observe that if  $v_k$ and $v_{k+i}$ are twins for $i=1,\dots,r-1$, then for $i\ne j\in\{1,\dots,r-1\}$,  $v_{k+i}$ and $v_{k+j}$ are twins of the same type as $v_k$ and $v_{k+i}$, because   $N(v_{k+i})=N(v_k)=N(v_{k+j})$ for independent twins and $N[v_{k+i}]=N[v_k]=N[v_{k+j}]$ for adjacent twins.

 It is useful to partition the vertices with one or more partition sets consisting of twins  and to use the partition to create block matrices, as in the proofs of Proposition \ref{t:twin} and Theorem \ref{t:twin-quot}.  
We make no claim that the next proposition is new but include the brief proof for completeness.
\begin{proposition}\label{t:twin} Let $G$ be a graph of order at least three 
 and suppose that $v_k$ and $v_{k+i}$ are twins for $i=1,\dots,r-1$.  {For $i=1,\dots,r-1$, let $\bz_i=[0,\dots,1 , 0,\dots,0,-1 ,0 , \dots , 0]^T$ be the vector where the $k$th coordinate is $1$ and the $k+i$th coordinate is $-1$. Then for $i=1,\dots,r-1$, $\bz_i$ is an eigenvector for {$A(G)$} for eigenvalue  $\alpha=0$ if $v_k$ and $v_{k+i}$ are independent or   $\alpha=-1$ if $v_k$ and $v_{k+i}$ are adjacent.  Thus eigenvalue} $\alpha$ has multiplicity at least $r-1$.
 \end{proposition}
 \bpf
 {We show that   $\bw=\bz_1$} is an eigenvector for eigenvalue $\alpha=0$ or $\alpha=-1$ of $A=A(G)$ where $v_1$ and $v_{2}$ are independent or  adjacent twins (the argument is the same for $v_k$ and $v_{k+i}$ but the notation is messier).  

Suppose $v_1$ and $v_{2}$ are independent twins.   Apply the  partition $\{1,2\}, \{3,\dots,n\}$ to {$A$} and $\bw$ to define block matrices and multiply:\vspace{-5pt}
\[A\bw=\mtx{A_{1,1}& A_{2,1}^T\\A_{2,1}&A_{2,2}}\mtx{\bw_1\\ \bw_2}=\mtx{A_{1,1}\bw_1+{A_{2,1}^T} \bw_2\\A_{2,1}\bw_1+A_{2,2} \bw_2}\vspace{-5pt}\]
Since $A_{1,1}=\mtx{0 & 0 \\0 & 0}$, $A_{2,1}=\mtx{\bv & \bv}$ for some vector $\bv$,  $\bw_1=\mtx{1\\-1}$, and $\bw_2=\bzero_{n-2}$,  \vspace{-5pt}
\[A_{1,1}\bw_1+{A_{2,1}^T} \bw_2=\bzero_2+\bzero_2=0\bw_1\mbox{ and }
A_{2,1}\bw_1+A_{2,2} \bw_2=\bzero_{n-2}+\bzero_{n-2}=0\bw_2.\vspace{-5pt} \]
Thus $A\bw={0}\bw$.

Suppose $v_1$ and $v_{2}$ are adjacent twins.   Apply the partition $\{1,2\}, \{3,\dots,n\}$ to $A=A(G)$ and $\bw$ to define block matrices and multiply:\vspace{-5pt}
\[A\bw=\mtx{A_{1,1}& A_{2,1}^T\\A_{2,1}&A_{2,2}}\mtx{\bw_1\\ \bw_2}=\mtx{A_{1,1}\bw_1+{ A_{2,1}^T} \bw_2\\A_{2,1}\bw_1+A_{2,2} \bw_2}\vspace{-5pt}\]
Since $A_{1,1}=\mtx{0 & 1 \\1 & 0}$, $A_{2,1}=\mtx{\bv & \bv}$ for some vector $\bv$,  $\bw_1=\mtx{1\\-1}$, and $\bw_2=\bzero_{n-2}$,  \[A_{1,1}\bw_1+{ A_{2,1}^T} \bw_2=\mtx{-1\\1}+\bzero_2=(-1)\bw_1 \mbox{ and} 
A_{2,1}\bw_1+A_{2,2} \bw_2=\bzero_{n-2}+\bzero_{n-2}=(-1)\bw_2. \]
Thus $A\bw=(-1)\bw$.
\epf

Sets of twins in a graph naturally provide an equitable partition of  {the adjacency matrix.} 
Proposition  \ref{t:twin} and Lemma \ref{t:quot} can be combined to determine the spectrum.

\begin{theorem}\label{t:twin-quot} Let {$G$ be a graph with $V(G)=\{1,\dots,n\}$, let $X=(X_1,\dots,X_p)$  be a partition of the vertices of $G$  with   $|X_1|\le \dots\le |X_p|$, and} let $k$ be the least index such that $|X_k|\ge 2$.  Suppose that $v,u\in X_j$ implies $v=u$ or $v$ and $u$ are twins.  For $j=k,\dots,p$, let  $\alpha_j$ denote the the eigenvalue $\alpha$ specified in Proposition \ref{t:twin} for  the type of twin in $X_j$. Let $B$ denote the quotient matrix of $ A(G)$ for $X$.  Then\vspace{-5pt}  \[\spec(A(G))=\{\alpha_k^{(n_k-1)},\dots,\alpha_p^{(n_p-1)}\}\cup\spec(B)\vspace{-5pt}\] (as multisets).
\end{theorem}
\bpf
{Let $n_i=|X_i|$ for $i=1,\dots,p$.} Apply Proposition  \ref{t:twin} to construct $n_j-1$ eigenvectors for $\alpha_j$, $j=k,\dots,p$ and denote this entire collection of  eigenvectors by  $\bw_1,\dots,\bw_{n-p}$; let $W_j$ denote the span of the subset of these vectors that are associated with $\alpha_j$.  It is immediate that $\{\alpha_k^{(n_k-1)},\dots,\alpha_p^{(n_p-1)}\}\subset\spec(A(G))$ (as multisets). By Lemma \ref{t:quot}, every eigenvector $\bz$ of $B$ for eigenvalue $\alpha$ yields an eigenvector $S\bz$ of $A(G)$ for $\alpha$.  Furthermore, $S\bz$ is orthogonal to (and thus independent of) $\bw_1,\dots,\bw_{n-m}$. Hence it suffices to show that $B$ has a basis of eigenvectors.

 Extend $\{\bw_1,\dots,\bw_{n-m}\}$ to a basis of eigenvectors \[\{\bw_1,\dots,\bw_{n-m},\bw_{n-m+1},\dots,\bw_{n}\}\] of {$A(G)$} (a basis of eigenvectors exists because $A(G)$ is symmetric). Consider $\bw_h$ with $h>n-m$.  If the associated eigenvalue $\alpha_h$ of $M$ is distinct from $\alpha_j$, then $\bw_h$ is orthogonal to the eigenvectors for $\alpha_j$.  If $\alpha_h=\alpha_j$, then let $\bw'_h=\bw_h-\proj_{W_j}(\bw_h)$ (this step can be applied more than once if needed).  Then $\bw'_h$ is an eigenvector for  $\alpha_h$ and is orthogonal to  $\bw_\ell$ for   $\ell=1,\dots,n-p$.  This implies $\bw'_h$ is constant on
the  coordinates in $X_j$ for $j=1,\dots,m$, so $\bw'_h=S\bz_h$ for some $p$-vector $\bz$.  By Lemma \ref{t:quot}, $\bz_h$ is an eigenvector for $B$ for $\alpha_h$. Thus  $B$ has a basis of eigenvectors and $\spec(A(G))=\{\alpha_k^{(n_k-1)},\dots,$ $\alpha_p^{(n_p-1)}\}\cup \spec(B)$  (as multisets). \epf

\section{Results on graph classes}\label{sec:classesresults}

In this section we  use the  tools obtained in the preceding section  and other known results  to establish Conjecture \ref{conj:energies} for several graph classes.

\subsection{Extended barbell graphs}\label{sec:extendedbarbell}

$\null$
{A \emph{barbell graph} is a graph composed of two cliques of the same size connected by an edge.  It was shown in \cite{EFGW16} that Conjecture \ref{conj:energies} is true for barbell graphs.  Here we apply the results of Section \ref{sec:eqot-twins} to establish Conjecture \ref{conj:energies} for extended barbell graphs.} 
An \emph{extended barbell graph} is a graph composed of two cliques of the same size, say of size $k$, connected by a path $P$ of length \textcolor{red}{$2.$}
+We label the vertices of the cliques by $v_1,\ldots, v_{k}$ and $v_{k+1}, \ldots, v_{n-1}$, and the {degree-two} vertex of the path by $v_{n}$, so that ${V(P)} = \{v_{k}, v_{n}, v_{k+1}\}.$ 

\begin{proposition} \label{p:extbarbell}
Let $G$ be an extended barbell graph on $n=2k+1$ vertices for $k \geq 3.$ 
The eigenvalues of $G$ are {$k-1$,  $-1$ with multiplicity {$n-4$},} and the three roots of $ f(x) = x^{3} - (k-2)x^2 - (1+k)x + 2(k-2).$ 
\end{proposition}

\bpf
 Let $G$ be an extended barbell graph and let $X=(X_1,X_2,X_3,X_4,X_5)$ be the partition of $V(G)$ defined by $X_1=\{v_{k}\}, X_2=\{v_{k+1}\}, X_3=\{v_{n}\}, X_4 = \{v_1, \dots, v_{k-1}\},$ and $X_{5} = \{v_{k+2}, \dots, v_{n-1}\}$. 
The partition $X$ is equitable with quotient matrix 
\[ B = \mtx{
   0 & 0 & 1 & k-1 & 0\\
   0 & 0 & 1 & 0   & k-1\\
   1 & 1 & 0 & 0   & 0\\
   1 & 0 & 0 & k-2 & 0\\
   0 & 1 & 0 & 0   & k-2
},
\]
and its characteristic polynomial is given by $p_{B}(x) = f(x) (x-k+1)(x + 1)$
where $ f(x) = x^{3} - (k-2)x^2 - (1+k)x + 2(k-2).$ 
Since the twins are adjacent, $\spec(G)=\{(-1)^{(2k-4)}\}\cup\spec(B)$ by  Theorem \ref{t:twin-quot}. This establishes that the eigenvalues of $G$ are $-1$ with multiplicity $2k-3=n-4$, $k-1$, and the three roots of $ f(x) = x^{3} - (k-2)x^2 - (1+k)x + 2(k-2).$ %
\epf

\begin{theorem}
Let $G$ be an extended barbell graph on $n=2k+1$ vertices for $k \geq 3.$ Then  {$\lambda_2 = {k-1}$, $ \lambda_{4} =\cdots =  \lambda_{n-1} = -1,$}  
$s^{+}(G) \geq n-1$, and $s^{-}(G) \geq n-1.$
\end{theorem} 
\bpf
By Proposition \ref{p:extbarbell}, the eigenvalues of $G$ are $-1$ with multiplicity $2k-3=n-4$, $k-1$, and the three roots {$\mu_1\ge \mu_2\ge \mu_3$} of $ f(x) = x^{3} - (k-2)x^2 - (1+k)x + 2(k-2).$ Since $f(k-1)=-2<0$ {and $f(-1)=2k-2>0$, $\mu_1>k-1>\mu_2>-1>\mu_3$.  Thus $\lam_1(G)=\mu_1$, $\lambda_2(G)=k-1$,  $\lam_3(G)=\mu_2$, $ \lambda_{4}(G) =\cdots =  \lambda_{n-1}(G) = -1$ and $\lam_{n}(G)=\mu_3$}.

For $s^{+}(G)$, we have that
\[
    s^{+}(G) \ge \lambda_1^{2}(G) + \lambda_2^{2}(G) 
     > 2(k-1)^2=2k^2-4k+2. \]
This implies that $s^{+}(G) > n-1=2k$ since $k  \geq 3.$ 

For $s^{-}(G)$ we have that
\[
    s^{-}(G) \ge \sum_{j=1}^{n-4} (-1)^{2} + \lambda_{n}^{2}(G) 
     = n-4 + \lambda_{n}^{2}(G).
\]
Since $f(-\frac 9 5)=\frac{14}{25}k-\frac{194}{125}$, for $k\ge 3$ we have $f(-\frac 9 5)\ge \frac{16}{125}>0$.  This implies $\lam_{n}(G)<-\frac 9 5$ and thus
\[
    s^{-}(G) > 
      n-4 + \lp -\frac 9 5\rp^{2} =n-4+3.24>n-1.
\]
\epf

\subsection{Unicyclic graphs}\label{subsec:unicyclic}

A \emph{unicyclic graph} is a connected graph that has exactly one cycle. In this section we apply  results from Section   \ref{sec:haemersinterlacing} and results established in other papers
to unicyclic graphs. We begin by applying Lemma \ref{l:partition-quotient} to the family of unicyclic graphs $U_{n,3}$ obtained by adding an edge between two leaves of the star $K_{1,n-1}$ (a \emph{leaf} is a vertex of degree one); see Figure \ref{fig:Un3}.
   
   \begin{figure}[h!]
    \centering
    \includegraphics[scale=0.51]{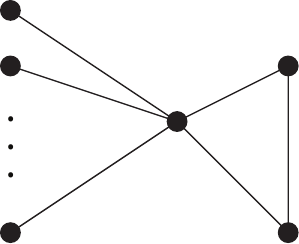}     
    \caption{The graph $U_{n,3}$. \label{fig:Un3}}
\end{figure}
   \begin{proposition} For $n\ge 3$, $s^+(U_{n,3}) \geq n-1$ and $s^-(U_{n,3}) \geq n-1.$
   \end{proposition}
   \bpf Observe first that $U_{n,3}$ contains a $K_{1,n-1}$ subgraph, so $s^+(U_{n,3})\ge n-1$.  Let $S$ be the set consisting of only the $n-1$-degree vertex of $U_{n,3}$.   Then in the notation of Lemma \ref{l:partition-quotient}, $s=1$, $n-s=n-1$, $c=n-1$, $d_1=0$ and $d_2=\frac 1{n-1}$. Since  $d_1d_2 - \frac{c^2}{(n-s)s} =-\frac{(n-1)^2}{n-1}< 0$, by Lemma \ref{l:partition-quotient},
  \bea
  \lambda_n(U_{n,3})^2 &\geq& \frac{1}{4} \left(\frac1{n-1} { -} \sqrt{\lp\frac1{n-1}\rp^2+ 4\frac{(n-1)^2}{n-1}} \right)^2\\ 
  &=&\frac{1}{4} \left(\lp\frac1{n-1}\rp^2 { -} \frac2{n-1}\sqrt{\lp\frac1{n-1}\rp^2+ 4(n-1)} +\lp\frac1{n-1}\rp^2+ 4(n-1)\right)\\
  &\ge & (n-1)-1.
  \eea
  Since the two vertices of degree two are adjacent twins, $-1\in\spec(U_{n,3})$ by  Proposition \ref{t:twin} and \[s^-(U_{n,3})\ge  \lambda_n(U_{n,3})^2+1=n-1.\]
   \epf

   Next we apply results of Guo and Spiro in  \cite{GS2022} to unicyclic graphs. 
 A \textit{homomorphism} from a graph $G$ to a graph $H$ is a map $\varphi:V(G)\to V(H)$ such that $\varphi(u)\varphi(v)\in E(H)$ whenever $uv\in E(G)$.    Observe that if $G$ is bipartite with vertex partition $V(G)=X\cup Y$ and $V(K_2)=\{1,2\}$, then $\varphi:V(G)\to V(K_2)$ defined by $\varphi(x)=1$ for $x\in X$ and $\varphi(y)=2$ for $y\in Y$ is  a homomorphism.   The \textit{Kneser graph} $\Kn(a,k)$ is  the graph whose vertices are the $k$-subsets of an $a$-element set, and two $k$-subsets are adjacent whenever they are disjoint. The \textit{fractional chromatic number} of a graph $G$ is given by
	\[\chi_f(G)=\inf_{(a,k)} \frac{a}{k},\]
where the infimum runs over all pairs $(a,k)$ such that there exist a homomorphism from $G$ to $\Kn(a,k)$. For more background on the fractional chromatic number, see \cite{Sch2011}.
{Guo and Spiro recently extended a bound of Ando and Lin \cite{AndLin2015} to the fractional chromatic number:

\begin{theorem}\label{GS22:frac-chrom}
 {\rm \cite{GS2022}}  For any graph $G$,
	\[\chi_f(G)\ge 1+\max\left\{\frac{s^+(G)}{s^-(G)},\frac{s^-(G)}{s^+(G)}\right\}. \]   
\end{theorem}}

   \begin{corollary}\label{cor:unicyclic-chi-f}
       {Let $G$ be} a unicyclic graph on $n$ vertices with an odd cycle of order $2m+1$ where $m\geq 2$. Then $ s^+(G), s^-(G) \ge \frac{2m}{2m+1} n$. In particular,  ${s^+(C_{2m+1}), s^-(C_{2m+1})} \ge 2m  =|V(C_{2m+1})|-1$.   \end{corollary}

\begin{proof}
 Let $C$ be the cycle of the unicyclic graph $G$.  Since $C$ is a subgraph of $G$, we have that $C$ has a homomorphism to $G$. Furthermore, $G$ has a homomorphism to $C$: Since $G$ is unicyclic, $G=C\cup \lp\cup_{i}^k T_i\rp$ where $T_i$ is a tree and $|V(C)\cap V(T_i)|=1$; we denote the unique vertex in $V(C)\cap V(T_i)$ by $v_i$. Since each $T_i$ is bipartite, we have a homomorphism from $T_i$ to $C$ by mapping the partite class containing $v_i$ to $v_i$ and the other partite class to a neighbor of $v_i$ on the cycle.  Together these maps define a homomorphsim from $G$ to $C$.
By composition of homomorphisms, we see that $G$ and $C$ have equal fractional chromatic numbers. It is well-known that the fractional chromatic number of an odd cycle $C_{2m+1}$ is $2 + \frac{1}{m}$. Thus, the fractional chromatic number of  a unicyclic graph containing a cycle $C_{2m+1}$ is also $2 + \frac{1}{m}$.

Thus, we have that 
 \[ 
  2 + \frac{1}{m} \ge 1+\max\left\{\frac{s^+(G)}{s^-(G)},\frac{s^-(G)}{s^+(G)}\right\},\quad \mbox{
   which implies }\quad
      \frac{m+1}{m} \ge \frac{s^+(G)}{s^-(G)} \quad \text{and} \quad   \frac{m+1}{m} \ge \frac{s^-(G)}{s^+(G)}.   \]

Since $s^+(G) + s^-(G) = 2 |E(G)| = 2n$, we can substitute $s^+(G) = 2n - s^-(G)$ into the first expression and obtain
\bea
     (m+1)s^-(G)  &\ge& m s^+(G)  = m (2n-s^-) \\
     (m+1 + m )s^-(G) &\ge& 2mn \\
     s^-(G) &\ge& \frac{2mn}{2m+1}. 
\eea
{By a similar} argument, we also obtain $  s^+(G) \ge \frac{2mn}{2m+1} $.
\end{proof}

We note that the conjecture was shown to be true for all regular graphs except for odd cycles in Theorem 8 in \cite{EFGW16}. 
It was claimed there that the conjecture was also true for odd cycles but no proof was presented.  Thus Corollary \ref{cor:unicyclic-chi-f} resolves the last regular graph case.
We note that, except in the case of the odd cycle, Corollary \ref{cor:unicyclic-chi-f} does not resolve Conjecture \ref{conj:energies} for the class of unicyclic graphs since $(\frac{2m}{2m+1} ) n \geq n-1$ only when $n \leq 2m+1$.

Using a stronger theorem from \cite{GS2022}, we can show that  unicyclic graphs containing a long odd cycle also satisfy the conjecture.

\begin{theorem}{\rm {\cite{GS2022}}}
If $G$ has a homomorphism to an edge-transitive graph $H$, then 
    \[\frac{\lambda_{\max}(H)}{|\lambda_{\min}(H)|}\ge \max\left\{\frac{s^+(G)}{s^-(G)},\frac{s^-(G)}{s^+(G)}\right\},\]
	where $\lambda_{\max}(H),\lam_{\min}(H)$ denote the greatest and least eigenvalue of $H$, respectively.
\end{theorem}
\begin{theorem}\label{lem:unicylic-big-m}
Let $G$ be a unicylic graph of order $n$ with an odd cycle of length $2m+1$ such that 
\[
m \geq \frac{\pi}{2\arccos\lp \frac{n-1}{n+1}\rp} - \frac{1}{2}. 
\]
Then $s^+(G) \geq n-1$ and $s^-(G) \geq n-1.$
\end{theorem}

\begin{proof}   Since $C_{2m+1}$ is edge-transitive and there exists a homomorphism from $G$ to $C_{2m+1}$,  by  Lemma \ref{lem:unicylic-big-m} we have that
	\[\frac{\lambda_{\max}(C_{2m+1})}{|\lambda_{\min}(C_{2m+1})|}\ge \max\left\{\frac{s^+(G)}{s^-(G)},\frac{s^-(G)}{s^+(G)}\right\}.\]
The eigenvalues of $C_{2m+1}$ are $2\cos \left( \frac{2\pi j }{2m+1}\right)$ for $j = 0,\ldots, 2m$. The largest eigenvalue is equal to $2$ and the least eigenvalue is 
\[
2\cos \left( \frac{2\pi m }{2m+1}\right) = -2\cos \frac{\pi}{2m+1}. 
\]
Thus, we have that 
\[
\frac{2}{2\cos \frac{\pi}{2m+1}} =\frac{1}{\cos \frac{\pi}{2m+1}} \ge \max\left\{\frac{s^+(G)}{s^-(G)},\frac{s^-(G)}{s^+(G)}\right\}.
\]
Since $s^+(G) + s^-(G) = 2|E(G)|= 2n$, we can rearrange to obtain that 
\[
s^-(G) \geq \frac{2 n \cos \frac{\pi}{2m+1}}{1 + \cos \frac{\pi}{2m+1}} , \quad s^+(G) \geq \frac{2 n \cos \frac{\pi}{2m+1}}{1 + \cos \frac{\pi}{2m+1}} .
\]
Let $m_0 = \frac{\pi}{2\arccos \frac{n-1}{n+1}} - \frac{1}{2}$. Then, we have that $\arccos \frac{n-1}{n+1} = \frac{\pi}{2m_0 + 1} $ and so
 $
\cos  \frac{\pi}{2m_0+1} = \frac{n-1}{n+1}.
$
Thus
\[ \frac{2 n \cos \frac{\pi}{2 m_0+1}}{1 + \cos \frac{\pi}{2 m_0+1}}{ =\frac{2n\frac{n-1}{n+1}}{\frac{n+1}{n+1}+\frac{n-1}{n+1}}  =\frac{2n(n-1)}{2n}}  = n-1 .
\]
Since $\cos  \frac{\pi}{2x+1}$ increases as $x$ increases, we obtain that 
\[\cos  \frac{\pi}{2m+1} \ge \frac{n-1}{n+1}
\]
for $m\ge m_0$.
\end{proof}

To give an idea of this bound, for a unicyclic graph on $100$ vertices, we need $m \ge 7.38$ for the Lemma to apply. Figure \ref{fig:nm-lem-unicyclic} shows a plot of $(n,m)$ where 
$
 m =\frac{\pi}{2\arccos {\frac{n-1}{n+1}}} - \frac{1}{2} , 
$ as in Lemma \ref{lem:unicylic-big-m}.
\begin{figure}[htbp]
    \centering
    \includegraphics[scale=0.4]{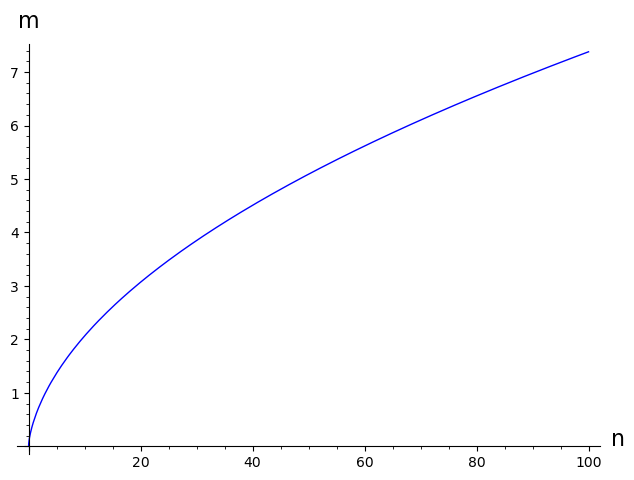}
    \caption{Plot of $n$ and $m(n) =\frac{\pi}{2\arccos  \frac{n-1}{n+1}} - \frac{1}{2}$, as in Lemma \ref{lem:unicylic-big-m}. 
    \label{fig:nm-lem-unicyclic} }
\end{figure}

\subsection{Graphs with two  positive eigenvalues}

In this section we show that if $G$ is a graph with exactly two positive eigenvalues, then $s^+(G)\geq s^-(G)$ and thus $s^+(G)\geq n-1$. Conjecture \ref{conj:energies} was established in  \cite{EFGW16} for    every graph that has exactly one positive eigenvalue, one negative eigenvalue, or (two negative eigenvalues and minimum degree at least two).  Since  $\lambda_1\ge \lambda_n$ for every graph, $\nu_G=1$ implies $\pi_G=1$ and thus $\pi_G=2$ implies $\nu_G\ge 2$.

\begin{proposition}\label{p:2pos}
Let $G$ be a connected graph of order $n\ge 4$ with $\pi_G=2$ positive eigenvalues.
Define  $\mu_i=\lambda_i, i=1,2$, $\mu_i=0, i=3,\dots,\nu_G$, and $\theta_i=|\lambda_{n+1-i}|, i=1,\dots,\nu_G$.   Then the  positive eigenvalues majorize the (reordered) absolute values of the negative eigenvalues, i.e.,
$
\sum_{i=1}^{k}\mu_{i}\geq\sum_{i=1}^{k}{\theta_{i}}
$
for all $k\leq \nu_G-1$ and $\sum_{i=1}^{\nu_G}\mu_{i}=\sum_{i=1}^{\nu_G}{\theta_{i}}$.  Whenever the  positive eigenvalues majorize the (reordered) absolute values of the negative eigenvalues, $s^+(G)\geq s^-(G)$ and $s^+(G)\geq n-1$.
\end{proposition}
\bpf Note that $\sum_{i=1}^{k}\mu_{i}=\sum_{i=1}^{k}{\theta_{i}}$ since $\tr A(G)=0$.  Since $\mu_1\ge \theta_1$, the  positive eigenvalues majorize the reordered absolute values of the negative eigenvalues.   This implies $s^+(G)\geq s^-(G)$  by Karamata's inequality, and so $s^+(G)\geq n-1$. \epf

We can apply Proposition \ref{p:2pos} to show that $s^+(H_n^3)\ge s^-(H_n^3)$ and thus $s^+(H_n^3)\ge n-1$ for the family of unicyclic graphs $H_n^3$ defined as follows:
For $n\ge k+2$, define  $H^k_n$ to be the graph obtained from $K_{1,n-k}$ and $C_k$ by identifying a degree $1$ vertex of $K_{1,n-k}$ with a vertex of $C_k$; $H^k_n$ has $n$ vertices. The graph in Figure \ref{fig:H9-3} is $H^3_9$.  
The graphs $H^3_n$ appear to minimize $s^-(H_n^3)$ over graphs of order $n$, as discussed in Section \ref{s:conclude}.

\begin{figure}[htbp]
    \centering
    \includegraphics[scale=0.5]{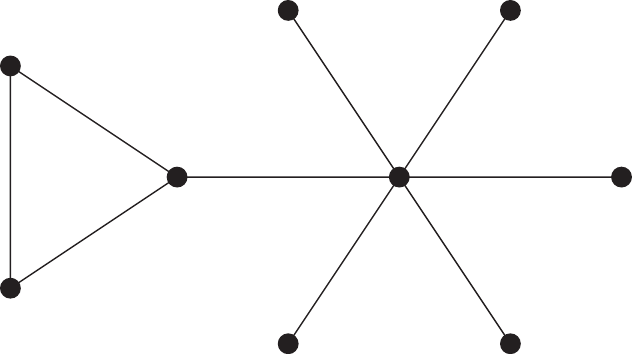}
    \caption{The unicyclic graph $H_9^3$}
    \label{fig:H9-3}
\end{figure}

\begin{proposition}\label{p:Hn3}
The graph $H_n^3$ has exactly two positive eigenvalues and three negative eigenvalues.  Thus  $s^+(H_n^3)\geq s^-(H_n^3)$  and $s^+(H_n^3)\geq n-1$.
\end{proposition}
\bpf The statements about $s^+(H_n^3)$ follow from Proposition \ref{p:2pos} once it is established that $H_n^3$ has exactly two positive eigenvalues.   It is verified computationally in \cite{sage} that $H_5^3$ has exactly two positive eigenvalues and three negative eigenvalues; specifically, $\spec(H_5^3)=\{2.214320, 1, -0.539189, -1, -1.675131\}$  (to six decimal places). Thus $\pi_{H_n^3}\ge 2$ and $\nu_{H_n^3}\ge 3$ by Lemma \ref{lemma:induced subgraph}.   Since $H_n^3$ has a set of $n-4$ independent twins, $0$ is an eigenvalue of $H_n^3$ with multiplicity $n-5$ by Proposition \ref{t:twin}. Thus $\pi_{H_n^3}= 2$ and $\nu_{H_n^3}= 3$.  
\epf
\subsection{Graphs with a certain {fraction} of positive (or negative) eigenvalues}\label{seubsec:fixednumberposevs}
In this section we utilize graph energy to  show that a sufficiently small percentage of the nonzero eigenvalues of a graph $G$ of order $n$ are positive (respectively, negative) then $s^+(G)\ge n-1$ (respectively, $s^-(G)\ge n-1$).

\begin{lemma}\label{l:energy}
Let $G$ be a graph with $\pi_G$  positive eigenvalues and $\nu_G$  negative eigenvalues. Then 
\[
s^+(G) \geq \frac{\cE(G)^2}{4\pi_G} \quad \text{ and } \quad  s^-(G) \geq  \frac{\cE(G)^2}{4\nu_G}. 
\]
\end{lemma}

\begin{proof}
Let $\lambda_1, \ldots, \lambda_{\pi_G}$ be the positive eigenvalues of $G$ and let $\lambda_{n - \nu_G + 1},\ldots, \lambda_{n}$ be the negative eigenvalues of $G$. Since the eigenvalues of $G$ sum to $0$, the energy of $G$ is as follows
\[
\cE(G) = \sum_{i =1}^{\pi_G} \lambda_i - \sum_{j=n - \nu_G + 1}^{n} \lambda_i = 2 \sum_{i =1}^{\pi_G} \lambda_i = - 2\sum_{j=n-\nu_G+1}^{n} \lambda_i. 
\]
By applying the Cauchy-Schwarz  inequality to the vector of positive eigenvalues and the all ones vector, we obtain
\[
\frac{\left(\frac{\cE(G)}{2}\right)^2}{\pi_G} = \frac{\left(\sum_{i =1}^{\pi_G} \lambda_i  \right)^2}{\pi_G} \leq \sum_{i =1}^{\pi_G} \lambda_i^2 = s^+(G).
\]
Similarly, we have 
\[
\frac{\left(\frac{\cE(G)}{2}\right)^2}{\nu_G} = \frac{\left(\sum_{i = n-\nu_G+1}^{n} \lambda_i  \right)^2}{\nu_G} \leq \sum_{i =n-\nu_G+1}^{n} \lambda_i^2 = s^-(G). 
\]
\end{proof}

 \begin{theorem}
\label{thm:atmostn/4posev}
Let  $G$ be a connected graph with $n\ge 3$ vertices.
If $\pi_G\leq \frac{(\rank A(G))^2}{4(n-1)}$, 
then $s^+(G)\geq n-1$. If $\nu_G \leq \frac{(\rank A(G))^2}{4(n-1)}$, then $s^-(G) \geq n-1$. 
\end{theorem}
\begin{proof} Let $r=\pi_G+\nu_G=\rank A(G)$.
 Let $a_r$ denote the product of the nonzero eigenvalues. Since $a_r=S_r(\lam_1,\dots,\lam_n)$ is the $r$th symmetric function of all the eigenvalues, $(-1)^ra_r$ is the coefficient  of $x{n-r}$ in the characteristic polynomial $p(x)$ of $A(G)$. Since all the entries of $A(G)$ are integers, every coefficient in $p(x)$ is an integer.  Since $a_r\ne 0$, this implies that $|a_r|\geq 1$.
By Lemma \ref {l:energy} and the arithmetic mean-geometric mean inequality,
\[
    {s^{ -}(G)}\geq \frac{ \cE(G)^{2}}{ 4\nu_G}=
    \frac{\lp |\lam_1|+\dots+|\lam_{\pi_G}|+|\lam_{n-\nu_G+1}|+\dots+|\lam_{n}|\rp^2}{ 4\nu_G}\ge
    \frac{\lb r\lp\prod_{\lam_i\ne 0}|\lambda_i|\rp^{\frac{1}{r}}\rb^2}{ 4\nu_G}\geq \frac{r^2}{4\nu_G}.\]
 If $\nu_G \leq \frac{r^2}{4(n-1)}$, then $s^-(G) \geq n-1$.
    
 An analogous argument shows the same statement for $s^+(G)$.
\end{proof}

\subsection{Cactus graphs}\label{s:cacti}
A \emph{cactus graph} is a connected graph in which any two  cycles have at most one vertex in common.  For a cactus graph $G$ that has sufficiently many even cycles relative to the number of odd cycles and its maximum degree,
we can delete vertices to obtain an induced bipartite graph and apply results from Section \ref{sec:interlacing} to conclude $G$ satisfies Conjecture \ref{conj:energies}.  While the  strategy of deleting vertices to obtain a bipartite graph (see Lemma \ref{l:del2bip} below) can be applied to any graph, it is particularly easy to use on cactus graphs.
The \emph{maximum degree} of a graph $G$ is $\Delta(G)=\max\{\deg v:v\in V(G)\}$.

\begin{lemma}\label{l:del2bip}   Let $G$ be a graph on $n$ vertices and let $S\subset V(G)$ be 
such that the subgraph of $G$ obtained by deleting the vertices in $S$ is bipartite. Then  
\[
s^+(G) \geq |E(G)| - |S|\,  \Delta(G)\text{ and }
s^-(G)  \geq|E(G)| - |S|\, \Delta(G).
\]
\end{lemma}

 \begin{proof} Let $H=G-S$, the graph induced by $V(G)\setminus S$.  Observe that deleing a single vertex can delete at most $\Delta(G)$ edges, so $E(H)\ge |E(G)|-|S|\Delta(G)$. Thus  by Corollary \ref{cor:induced subgraph},
 $s^+(G) \geq |E(G)| - |S|\,  \Delta(G)$ and $s^-(G)  \geq|E(G)| - |S|\, \Delta(G)$. 
 \end{proof} 

The next corollary applies Lemma \ref{l:del2bip} to cactus graphs.

\begin{corollary}\label{c:sminus}
 Let $G$ be a cactus graph on $n$ vertices with $k$ odd cycles and $\ell$ even cycles. If
$\ell \geq k(\Delta{(G)} -1)$, then $G$ satisfies Conjecture \ref{conj:energies}
\end{corollary}
\bpf It is easy to see that $|E(G)| = n-1 + k + \ell$. The deletion of $k$ vertices, one from each odd cycle, will result in a bipartite graph.  So if $\ell \geq k(\Delta{(G)} -1)$, then $|E(G)|=n-1+k+\ell\ge n-1+k\Delta(G)$ and the result follows from Lemma \ref{l:del2bip}.
\epf

However,  one can often break multiple odd cycles by deleting a single vertex, in which case applying Lemma \ref{l:del2bip} or Corollary \ref{cor:induced subgraph} directly is  preferred, as in the next example.

\begin{example}{\rm  Let $G$ be 
the cactus graph shown in Figure \ref{f:goodcactus}.  Deleting  the one vertex incident to the two 3-cycles results in a bipartite graph with $13=|V(G)|-1$ edges, so $G$ satisfies Conjecture  \ref{conj:energies} by Corollary \ref{cor:induced subgraph}.
 }\end{example}

\begin{figure}[!h]    \centering    \includegraphics[scale=0.45]{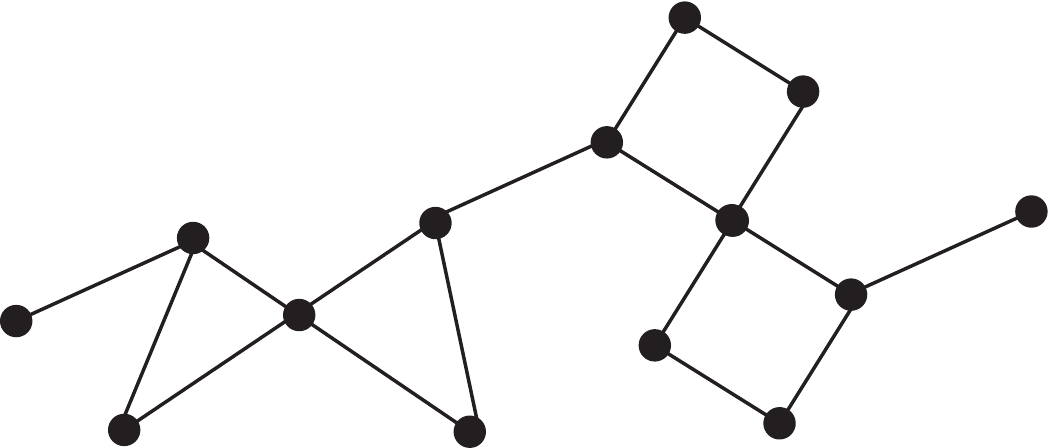}    \caption{A cactus to which Lemma \ref{l:del2bip} applies.\label{f:goodcactus}}  \end{figure}

\section{Concluding remarks and open problems}\label{s:conclude}
In this section, we discuss  some  questions whose solutions  may shed light on the conjecture  and   present related computational data.  In particular, we look at the relative magnitude of $s^+(G)$ and $s^-(G)$ and at some graph families that seem difficult for the conjecture.

From computations on small graphs, as summarized in Table \ref{tab:splus-vs-sminus}, we see that it is much more common that $s^+(G)$  is larger. 

\begin{table}[htbp]
    \centering
    \begin{tabular}{c|c|c|c|c|c|c}
n & \small \# graphs & \small \#  $s^+ > s^-$ & \small \#  $s^- > s^+$ & \small \% $s^- > s^+$ & \small \# bipartite & \small \#  $s^+ = s^-$\\
\hline
2 & 1 & 0 & 0 & 0.000000 & 1 & 1\\
3 & 2 & 1 & 0 & 0.000000 & 1 & 1\\
4 & 6 & 3 & 0 & 0.000000 & 3 & 3\\
5 & 21 & 15 & 1 & 4.76190 & 5 & 5\\
6 & 112 & 93 & 2 & 1.78571 & 17 & 17\\
7 & 853 & 795 & 14 & 1.64127 & 44 & 44\\
8 & 11117 & 10848 & 87 & 0.782585 & 182 & 182    \end{tabular}
    \caption{The number of graphs where $s^+(G) > s^-(G)$, $s^+(G) = s^-(G)$, and $s^-(G) > s^+(G)$  for connected graphs on up to $8$ vertices.\label{tab:splus-vs-sminus}}\vspace{-10pt}
\end{table}
}

\begin{question}
    Does the percentage of graphs $G$ with $s^-(G)>s^+(G)$ tend to zero as $n$ goes to infinity?
\end{question}

We also note that for $n \leq 8$, there were no non-bipartite graphs on $n$ vertices that had $s^+(G) = s^-(G)$. 
\begin{question}
    Do there exist non-biparite graphs for which $s^+(G) = s^-(G)$?
\end{question}

The class of non-bipartite unicyclic graphs seems to be a particularly difficult case for Conjecture \ref{conj:energies}.  In particular, the conjecture is still open for unicyclic graphs $G$ of order $n\ge 10$  such that $G$ has a  $3$-cycle and is not isomorphic to $U_{n,3}$.  It is not surprising that unicyclic graphs challenge the conjecture, since they are close to graphs that achieve equality in the bound: A connected unicyclic graph of order $n$ has  $n$ edges and a connected graph $T$ of order $n$ with $n-1$ edges is a tree and has $s^+(T)=s^-(T)=n-1$ (however, the graph with the maximum number of edges, $K_n$, also has $s^-(K_n)=n-1$).

We performed computations  
 on non-bipartite, connected unicyclic graphs (equivalently, connected graphs on $n$ vertices with $n$ edges) for $n =3,\ldots, 18$ \cite{sage}. We summarize the minimum value of $s^+(G)$ and $s^-(G)$ among these graphs, as well as the number of isomorphism classes of such graphs,  in Table \ref{tab:unicyclic}.
  Recall that $H^k_n$ is the graph of order $n$ obtained from $K_{1,n-k}$ and $C_k$ by identifying a degree $1$ vertex of $K_{1,n-k}$ with a vertex of $C_k$.  It was shown in Proposition \ref{p:Hn3} that $s^+(H^3_n)\ge n-1$.
By Theorem \ref{lem:unicylic-big-m},  Conjecture \ref{conj:energies} is true for $H^5_n$ for $n\le 48$.  For each case $n=3,\dots,18$ and each of $s^+(G)$ and $s ^-(G)$, the {minimum  value is attained by only one isomorphism class of graphs \cite{sage} (since $s^+(G)+s^-(G)=2|E(G)|=2n$ for a unicyclic graph, a minimizer for $s^+(G)$ is a maximizer for $s^-(G)$ and vice versa).}   In particular, the minimizer of $s^+(G)$  among non-bipartite unicyclic graphs of order $n =  7,\ldots, 18$ is $H_n^5$ and the minimizer of $s^-(G)$  among non-bipartite unicyclic graphs of order $n =  5, \ldots, 18$ is $H_n^3$.  This leads us to ask whether this is true in general.

 \begin{table}[htbp]
\small
    \centering
    \begin{tabular}{c|c|c|c|c|c|c|c|c|c}
 $n$ & 3 & 4 & 5 & 6 & 7 & 8 & 9 & 10 & 11 \\
 \hline
total graphs & 1 & 1 & 4 & 8 & 23 & 55 & 155 & 403 & 1116  \\
 $\min s^+(G)$ & 4.0 & 4.806063 & 4.763932 & 5.8548 & 6.797054 & 7.786641 & 8.78153 & 9.778404 & 10.776269\\
$\min s^-(G)$  & 2.0 & 3.193937 & 4.096788 & 5.073208 & 6.060343 & 7.051905 & 8.045829 & 9.041196 & 10.037521 \\
\end{tabular}
    
    \vspace{10pt}
    \begin{tabular}{c|c|c|c|c|c|c|c}
$n$  & 12 & 13 & 14 & 15 & 16 & 17 & 18  \\
\hline
total graphs & 3029 & 8417 & 23285 & 65137 & 182211 & 512625 & 1444444 \\
$\min  s^+(G)$  & 11.774708 & 12.773512 & 13.772564 & 14.771792 & 15.771151 & 16.77061 & 17.770146  \\
$\min  s^-(G)$  & 11.034519 & 12.032012 & 13.029882 & 14.028045 & 15.026442 & 16.025029 & 17.023774  \\
    \end{tabular}
    \caption{The minimum values of $s^+(G)$ and $s^-(G)$, rounded to $6$ decimal places, among non-bipartite unicyclic graphs of order $n$.}
    \label{tab:unicyclic}
\end{table}

\begin{question}\label{conj:unicyclic}
    Let $G$ be a non-bipartite unicyclic graph on $n \geq 19$ vertices.  Is  $s^+(G) \geq s^+(H^5_n)$?  Is $s^-(G) \geq s^-(H_{n}^3)$?
\end{question}

Next we describe some preliminary efforts to show that $s^-(H^3_n)\ge n-1$.  By Proposition \ref{t:twin}, $A(H^3_n)$ has eigenvalues $-1$  and $0$ with multiplicity $n-5$. From Proposition \ref{p:Hn3} we know that $\pi_{H^3_n}=2$ and $\nu_{H^3_n}=3$.  

In $H^3_n$, label the vertices as follows: The degree-2 vertices are 1 and 2, the degree-3 vertex is 3, the degree-($n-3$) vertex adjacent to one or more leaves is 4, and the leaves are $5,\dots,n$. The partition $X_1=\{1,2\}$, $X_2=\{3\}$,  $X_3=\{4\}$, $X_4=\{5,\dots,n\}$ is equitable.  The quotient matrix is 
\[B=\mtx{1 & 1 & 0 & 0\\
2 &0 & 1 & 0\\
0 & 1 & 0 & n-4\\
0 & 0 & 1 & 0}.\]
 The characteristic polynomial of $B$ is $p_B(x)=x^4 - x^3 - (n - 1) x^2 + (n - 3) x + 2 (n - 4)$.  Denote the eigenvalues of $B$ by $\mu_1>\mu_2>\mu_3>\mu_4$.   Thus
$\spec(A(H^3_n))=\{\mu_1,\mu_2,\mu_3,\mu_4,-1,0^{(n-5)}\}$. Since $A(H^3_n)$ has exactly two positive eigevalues, $\mu_2>0>\mu_3$.     
To prove the conjecture for $H^3_n$, it is suffices to show that $\mu_3^2+\mu_4^2\ge n-2$. Similar methods can be applied to $H^5_n$ (a $5\x 5$ quotient matrix can be obtained by using an equitable partition that groups the two cycle neighbors of the degree-3 cycle vertex together and groups the other two degree-2 cycle vertices together).
 
 Another interesting approach is the behavior of $s^+(G)$ and $s^-(G)$ when a leaf is added at a vertex of $G$. Let $v$ be a vertex of $G$ and let $G^v$ be obtained by adding a leaf adjacent to $v$. We can then look at the quantities
\[
s^+(G^v) - s^+(G), s^-(G^v) - s^-(G).
\]
Since $G^v$ has exactly one more vertex and one more edge than $G$, one might hope that the increment to $s^+(G)$ and $s^-(G)$ is at least one, but that is not the case, see Figure \ref{fig:unicyclic-pendant}. 

\begin{figure}[htbp]
    \centering
    \includegraphics[scale=0.7]{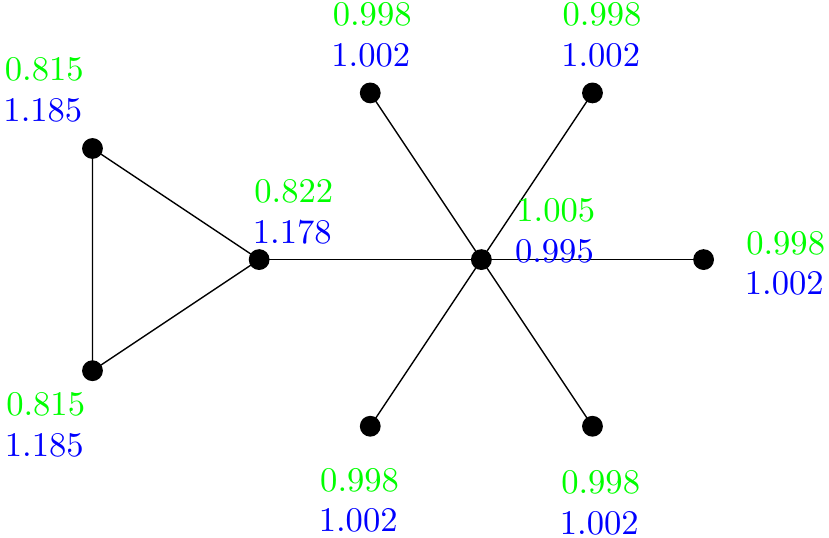}
    \caption{A unicyclic graph on $9$ vertices; each vertex $v$ is labelled by the increase to $s^+$ (in green) and to $s^-$ (in blue), resulting from adding a vertex adjacent to only $v$. }
    \label{fig:unicyclic-pendant}
\end{figure}


\bigskip
{\bf Acknowledgment.} This project started and was made possible by \textit{Spectral Graph and Hypergraph Theory: Connections \& Applications}, December 6-10, 2021, a workshop at the American Institute of Mathematics with  support from the US National Science Foundation. The authors thank AIM and also thank Sam Spiro for many fruitful discussions. 
Aida Abiad thanks Clive Elphick for bringing Conjecture \ref{conj:energies} to her attention. 
Aida Abiad is partially supported by the Dutch Research Council through the grant VI.Vidi.213.085 and by the Research Foundation Flanders through the grant 1285921N. Leonardo de Lima is partially supported by CNPq grant 315739/2021-5.

 
\end{document}